\RequirePackage{tikz}
\documentclass[sn-mathphys]{sn-jnl}

\usepackage{graphicx}% Include figure files
\usepackage{dcolumn}% Align table columns on decimal point
\usepackage{bm}% bold math
\usepackage{hyperref}% add hypertext capabilities
\usepackage[mathlines]{lineno}% Enable numbering of text and display math
\usepackage[T1]{fontenc}
\usepackage[english]{babel}
\usepackage{amsthm}
\usepackage{tabularx}
\usepackage{float}
\usepackage{amscd}
\usepackage{amsfonts}         
\usepackage{amsmath}
\usepackage{amssymb}
\usepackage{enumerate}
\usepackage{tikz}
\usepackage{tikz-cd}
\usepackage{mathrsfs}  

\jyear{2022}

\newtheorem{theo}{Tplottin ubuntuheorem}[section]
\theoremstyle{plain}
\newtheorem{thm}[theo]{Theorem}
\newtheorem{lem}[theo]{Lemma}

\newtheorem*{thm*}{Theorem}
\newtheorem*{lem*}{Lemma}
\newtheorem*{prop*}{Proposition}
\newtheorem*{cor*}{Corollary}

\theoremstyle{definition}
\newtheorem{defn}[theo]{Definition}

\newtheorem{cons}[theo]{Construction}
\newtheorem{rem}[theo]{Remark}

\newcommand{\Db}{\mathbb{D}}

\newcommand{\Nb}{\mathbb{N}}

\newcommand{\Rb}{\mathbb{R}}

\newcommand{\Zb}{\mathbb{Z}}

\newcommand{\Lc}{\mathcal{L}}

\newcommand{\Uc}{\mathcal{U}}

\newcommand{\Fs}{\mathscr{F}}

\newcommand{\Bfr}{\mathfrak{B}}

\newcommand{\hook}{\hookrightarrow}

\newcommand{\SO}{\mathrm{SO}}
\newcommand{\SU}{\mathrm{SU}}
\newcommand{\Id}{\mathrm{Id}}

\newcommand{\xto}{\xrightarrow}

\newcommand{\Dic}{\mathrm{Dic}}

\begin{document}

\title[Topologically protected tricolorings]{Bordism invariants of colored links and topologically protected tricolorings}

\author*[1,2]{\fnm{Toni} \sur{Annala}}\email{tannala@math.ias.edu}

\author[2]{\fnm{Hermanni} \sur{Rajamäki}}\email{hermanni.rajamaki@live.com}

\author[2]{\fnm{Mikko} \sur{Möttönen}}\email{mikko.mottonen@aalto.fi}

\affil*[1]{\orgdiv{School of Mathematics}, \orgname{Institute for Advanced Study}, \orgaddress{\street{1 Einstein Drive}, \city{Princeton}, \postcode{08540}, \state{NJ}, \country{USA}}}

\affil[2]{\orgdiv{QCD Labs, QTF Centre of Excellence,
Department of Applied Physics}, \orgname{Aalto University}, \orgaddress{\street{P.O. Box 13500}, \city{Aalto}, \postcode{FI-00076}, \country{Finland}}}


\abstract{We construct invariants of colored links using equivariant bordism groups of Conner and Floyd. We employ this bordism invariant to find the first examples of topological vortex knots, the knot structure of which is protected from decaying via topologically allowed local surgeries, i.e., by reconnections and strand crossings permitted by the topology of the vortex-supporting medium. Moreover, we show that, up to the aforementioned local surgeries, each tricolored link either decays into unlinked simple loops, or can be transformed into either a left-handed or a right-handed tricolored trefoil knot.}



\maketitle

\section{Introduction}\label{sect:intro}

The theory of topological defects in ordered media \cite{volovik:1977, mermin:1979, mineev:1998} provides an intriguing opportunity to employ the methods of geometry and topology in the study of concrete physical phenomena, which are often amenable to numerical and experimental investigation. Accordingly, topological defects have attracted persistent attention from scientists from various backgrounds. On the theoretical side, cosmic ``Alice'' strings in theoretical cosmology \cite{bucher:1992, bucher:1992b} provide a potential answer to the baryogenesis problem, while the zoo of liquid crystals of various types leads to a delightfully diverse set of geometric theories \cite{kamien:2000, beller:2014, machon:2019}. On the experimental side, recent successes include the experimental realizations of synthetic Dirac \cite{pietila:2009b, ray:2014} and isolated monopoles \cite{ray:2015}, as well as the realization of a quantum ``knot'' \cite{hall:2016}, which is a three-dimensional soliton that is closely related to the famous Hopf fibration.

Here, we investigate topological vortices in three-dimensional systems, i.e., singular defects, the singular locus of which are one dimensional. For simplicity, we assume that the singular locus is a disjoint union of smoothly embedded circles, i.e., a link. The local structure of such defects is well understood, as the local types of topological vortices correspond to conjugacy classes in $\pi_1(X)$, where $X$ is the \emph{order parameter space}, which is a topological space that parameterizes the possible local configurations the physical system may assume \cite{volovik:1977, mermin:1979}. However, the global defect structure is more subtle. Unknotted ring defects, and unlinked disjoint union thereof, may be classified by pairs of $\pi_1$ and $\pi_2$ charges, up to the action of $\pi_1$ \cite{nakanishi:1988,annala-mottonen}. However, most types of ordered media admit intricate knotted and linked vortex configurations. Even for one of the simplest cases, that of the nematic liquid crystal (for which $\pi_1(X) = C_2$), the classification is no simpler than the classification of all links up to isotopy \cite{machon:2013, machon:2014}.

In order to explain how to make the global defect structure more amenable to classification, we take a short historical detour. The idea of knotted and linked vortex loops is an old one, and may be traced back to Kelvin's vortex atom hypothesis \cite{thomson:1869}, which was based on the observation that the linking type of a vortex core remains unchanged in the dynamical evolution of a dissipationless ideal fluid. However, this assumption is violated in realistic situations, as the knotted vortex structure tends to decay via local reconnection events \cite{kleckner:2013, kleckner:2014, kleckner:2016}. Taking into account reconnection events in the context of topological vortices, as well as the crossing events of vortices that correspond to a pair of commuting elements in $\pi_1(X)$ \cite{poenaru:1977}, a framework for investigating topological vortex configurations up to such core-topology-altering moves, \emph{topologically allowed local surgeries}, was considered in Ref.~\cite{annala:2022}. In physical terms, this scheme models the intermediate-energy evolution of a topological vortex configuration: even though altering the core-topology requires energy \cite{monastyrsky:1986}, the cost is small if it is performed inside a volume that is much smaller than the size of the vortex loops \cite{annala:2022}.

Under the assumption that $\pi_2(X) \cong 0$, which is satisfied, for instance, by several phases of spinor Bose--Einstein condensates \cite{kawaguchi:2012}, as well as biaxial nematic liquid crystal \cite{volovik:1977, mermin:1979}, the data of a topological vortex configuration may be faithfully presented by a $G$-colored link diagram, where $G := \pi_1(X)$ \cite{fox:1970, annala:2022}. The core-topology-conserving evolution corresponds to $G$-colored Reidemeister moves, and the topologically allowed local surgeries admit simple diagrammatic interpretations as well \cite{annala:2022}. For small groups $G$, this data admits a particularly simple description in terms of link diagrams in which each arc is colored according to a coloring scheme, which is a generalization of the well-known concept of a \emph{tricoloring} of a knot \cite{fox:1970,crowell:1977}, in which the three colors correspond to the three inversions $\sigma_{ij}$ in the symmetric group $\Sigma_3$.

In the previous work \cite{annala:2022}, $Q_8$-colored links were classified up to topologically allowed local surgeries: essentially, there exists only three non-trivial classes. One should contrast this simple classification with the complicated infinitude of all possible links. However, because the quaternion group $Q_8$ does not support topologically protected knots, all the topologically protected structures discovered thus far are links. 

Here, we show that there is no mathematical obstruction for the existence of topologically protected knots in general: the symmetric group $\Sigma_3$ does support topologically protected knots (Theorem~\ref{thm:nontrivtrefoil}). Moreover, we prove that, using topologically allowed local surgeries, any topologically protected tricolored link may be transformed into either a left-handed or a right-handed tricolored trefoil knot (Theorem~\ref{thm:tricoloclassification}). Again, the classification is not only finite, but also remarkably simple.

In order to investigate topological protection of tricolored knots, we introduce in Section~\ref{sect:cobinv}, which is the technical heart of the article, the \emph{bordism invariant} of $(G,S)$-colored links. Here, $G$ is a finite group and $S \subset G$ is a subset that is closed under inverses and conjugation; a $(G,S)$-colored link is such a $G$-colored link, that the elements of $G$ that appear in the colored link diagram belong to $S$. The case of tricolorings corresponds to $G = \Sigma_3$ and $S = I := \{\sigma_{12}, \sigma_{13}, \sigma_{23} \}$, where $\sigma_{ij}$ is the inversion swapping the $i^{th}$ and the $j^{th}$ element. For each $(G,S)$-colored link, one may construct a branched cover $M$ of $S^3$ on which the group $G$ acts. It has the property that the nontrivial stabilizers of the $G$-action are cyclic, and generated by an element of $S$. Moreover, the space $M$ admits a canonical structure of an oriented manifold, and the bordism invariant of the link is defined to be the class of $M$ in the third $G$-equivariant bordism group, with a stabilizer condition given by $S$. Such groups have been studied previously by Conner and Floyd \cite{conner-floyd:1964, conner-floyd:1966}. 

The bordism invariant satisfies many pleasant properties (Theorem~\ref{thm:consofcobinv}): it is additive in untangled disjoint unions of links, and takes value 0 on simple loops. Moreover, it is conserved in topologically allowed local reconnections. For tricolorings, all topologically allowed local surgeries may be expressed in terms of reconnections and tricolored Reidemeister moves, and therefore the bordism invariant may be employed to detect topological protection of a tricolored link. Even though we do not apply the bordism invariant in greater generality here, we believe it to provide a valuable tool for future efforts in the study of colored links.

\subsubsection*{Conventions}
Throughout the article, all links are tame. Unless otherwise stated, all manifolds are smooth and closed. Cyclic groups are denoted either by $C_n$ or $\Zb_n$, depending if they are considered as a group or an Abelian group. The symmetric group on $n$ letters is denoted by $\Sigma_n$.

\subsubsection*{Acknowledgments}

We thank Keegan Boyle, Cihan Okay, Ben Williams, and Roberto Zamora Zamora for interest in this work as well as useful discussions and questions. We have received funding from the European Research Council under Grant No. 681311 (QUESS) and from the Academy of Finland Centre of Excellence program (Project No. 336810), as well as from the Vilho, Yrjö and Kalle Väisälä Foundation of the Finnish Academy of Science and Letters. 

\section{Bordism invariants of colored links}\label{sect:cobinv}

Here, we construct bordism invariants for colored links, which we will employ in Section~\ref{sect:tricolor} to study topologically protected tricolorings. The idea behind the invariant is quite simple: for a $G$-colored link, we associate an oriented manifold $M$ with an action of the group $G$. We prove that the equivariant oriented bordism class of $M$ is conserved in topologically allowed local reconnections. Hence, we obtain an invariant of $G$-colored links that is conserved in topologically allowed local reconnections.

Let us recall necessary background \cite{fox:1970, annala:2022}. If $G$ is a finite group, then a \emph{$G$-colored link} is a pair $(\Lc,\psi)$, where $\Lc \subset S^3$ is a link and the \emph{$G$-coloring} $\psi$ is a group homomorphism $\pi_1(S^3 \backslash \Lc) \to G$. Two $G$-colorings $\psi, \psi': \pi_1(S^3 \backslash \Lc) \to G$ are considered \emph{equivalent} if $\psi'$ can be obtained from $\psi$ by conjugating it with an element of $G$. The equivalence classes of $G$-colorings $\psi$ on a given link $\Lc$ are in a bijective correspondence with the set of homotopy classes $[\Rb^3 \backslash \Lc, X]$, where $X$ is any path-connected space satisfying $\pi_1(X) \cong G$ and $\pi_2(X) \cong 0$ (e.g. the $X$ is the classifying space $\mathrm{B}G$ \cite{hatcher:2002}). In physical terms, a $G$-colored link captures the homotopical information of an $X$-valued order parameter field on the complement of the link $L$ in $\Rb^3$. The data of a $G$-colored link may be depicted by a $G$-colored link diagram (Fig.~\ref{fig:linkbasics}).

\begin{cons}[Branched covering $M(\Lc, \psi)$ of a colored link]\label{cons:M}
Let $(\Lc, \psi)$ be a $G$-colored link, and denote $U(\Lc) := S^3 \backslash \Lc$. Let $\tilde U(\Lc, \psi)$ be the principal $G$-bundle over $U(\Lc)$, with monodromy specified by $\psi$ \cite{hatcher:2002}. Note that $\tilde U(\Lc, \psi)$ depends only on the equivalence class of $\psi$. Denote by $M(\Lc, \psi)$ the unique completion of $\tilde U(\Lc, \psi)$ to a branched covering of $S^3$ \cite{fox:1957, fox:1970}. The space $M(\Lc, \psi)$ has a canonical structure of an oriented smooth manifold, and the group $G$ acts on it by orientation-preserving diffeomorphisms. An illustration of the local structure of the branched covering $M(\Lc, \psi) \to S^3$ is provided in Fig.~\ref{fig:bcover}.
\end{cons}

The stabilizers of the $G$-action on $M(\Lc, \psi)$ are cyclic.

\begin{lem}\label{lem:stabilizers1}
Let $(\Lc, \psi)$ be a colored link, and let $L$ be component of $\Lc$. Suppose that the (free) homotopy class of a meridian\footnote{A small loop that winds about $L$ once.} of $L$ maps to the conjugacy class of $g \in G$ under the $G$-coloring $\pi_1(S^3 \backslash \Lc) \to G$. Then the stabilizer groups of the points in the preimage of $L$ are conjugate to the cyclic group $\langle g \rangle$ generated by $g$.
\end{lem}
\begin{proof}
This is evident from Fig.~\ref{fig:bcover}.
\end{proof}

\begin{defn}\label{def:cobinv}
Let $G$ be a finite group and $S \subset G$ be a subset of \emph{allowed stabilizers}, which is assumed to be closed under inverses and conjugation with elements of $G$. A $(G,S)$-colored link is a $G$-colored link $(\Lc, \psi)$ such that the $G$-coloring sends meridians of $\Lc$ to elements in $S$. Given a $(G,S)$-colored link $(\Lc, \psi)$, we associate to it the \emph{bordism invariant}
\begin{equation}
\Bfr_{G,S}(\Lc, \psi) := [M(\Lc, \psi)] \in \Omega^{\SO}_3(G, \Fs_S),
\end{equation}
where $\Omega^{\SO}_*(G, \Fs_S)$
denotes the group of isomorphism classes of closed oriented smooth $G$-manifolds $X$, such that the stabilizer subgroup of each point of $X$ belongs to $\Fs_S := \{\langle g^n \rangle \subset G \vert g \in S, n \in \Nb\}$\footnote{The role of the exponent $n$ is to verify that the family $\Fs_S$ is closed under inclusions, a condition assumed by Conner and Floyd \cite{conner-floyd:1966}.}, up to oriented $G$-equivariant bordisms, the stabilizers of which are in $\Fs_S$ \cite[Section 5]{conner-floyd:1966}.
\end{defn}

From now on, $G$ denotes a finite group, and $S \subset G$ is a subset that is closed under inverses and conjugation. Let us study the basic properties of the bordism invariant $\Bfr_{G,S}$.

\begin{defn}
 A link $\Lc \subset S^3$ is an \emph{untangled disjoint union} of links $\Lc_1, \Lc_2 \subset S^3$, denoted by 
\begin{equation}
\Lc = \Lc_1 \uplus \Lc_2,
\end{equation}
if $\Lc$ is a disjoint union of $\Lc_i \in S^3$, and the links $\Lc_i$ can be separated by a smooth embedded 2-sphere $\Omega \in S^3$. If $\Lc = \Lc_1 \uplus \Lc_2$, $\Omega$ is a separating two-sphere, and $x \in \Omega$, then van Kampen's theorem \cite{hatcher:2002} provides a canonical identification
\begin{equation}
\pi_1(S^3 \backslash \Lc, x) = \pi_1(S^3 \backslash \Lc_1, x) * \pi_1(S^3 \backslash \Lc_2,x),
\end{equation}
where $*$ stands for the free product of groups. A colored link $(\Lc, \psi)$ is an \emph{untangled disjoint union} of colored links $(\Lc_1, \psi_1)$ and $(\Lc_2, \psi_2)$, denoted by
\begin{equation}
(\Lc,\psi) = (\Lc_1, \psi_1) \uplus (\Lc_2, \psi_2),
\end{equation}
if $\Lc = \Lc_1 \uplus \Lc_2$, and, given a separating two-sphere $\Omega \subset S^3$,
\begin{equation}
\psi = \psi_1 * \psi_2
\end{equation}
as group homomorphisms $\pi_1(S^3 \backslash \Lc, x) \to G$.
\end{defn}

\begin{lem}\label{lem:addcobinv}
The bordism invariant $\Bfr_{G,S}(\Lc, \psi)$ is an \emph{additive invariant} in the sense that, if $(\Lc, \psi) = (\Lc_1, \psi_1) \uplus (\Lc_2, \psi_2)$, then
\begin{equation}
\Bfr_{G,S}(\Lc, \psi) = \Bfr_{G,S}(\Lc_1, \psi_1) + \Bfr_{G,S}(\Lc_2, \psi_2)
\end{equation}  
in $\Omega^\SO_3(G, \Fs_S)$.
\end{lem}
\begin{proof}
If $\Lc = \Lc_1 \uplus \Lc_2$, then the manifold $S^3 \backslash \Lc$ is diffeomorphic to the connected sum $(S^3 \backslash \Lc_1) \# (S^3 \backslash \Lc_2)$. Let $P$ be the standard smooth oriented bordism from $S^3 \coprod S^3$ to $S^3 \# S^3 \cong S^3$ (a higher dimensional version of the \emph{pair of pants}), and $P^\circ = P \backslash (I \times \Lc_1 \coprod I \times \Lc_2)$ the induced oriented bordism from $(S^3 \backslash \Lc_1) \coprod (S^3 \backslash \Lc_2)$ to $S^3 \backslash \Lc$. By van Kampen's theorem and homotopy invariance, $\pi_1(P^\circ) \cong \pi_1(S^3 \backslash \Lc_1) * \pi_1(S^3 \backslash \Lc_2)$, and $\psi_1$ and $\psi_2$ define a group homomorphism $\tilde \psi: \pi_1(P^\circ) \to G$. Let $W$ be the completion of the principal $G$-bundle $W^\circ$ over $P^\circ$, defined by the group homomorphism $\tilde \psi$, to a branched cover over $P$. The total space $W$ has a natural structure of an oriented smooth manifold with boundary. The group $G$ acts on it by orientation preserving diffeomorphisms, and therefore $W$ an equivariant oriented bordism from $M(\Lc_1, \psi_1) \coprod M(\Lc_2, \psi_2)$ to $M(\Lc, \psi)$. The non-trivial stabilizers of the $G$-action on $W$ belong to $\Fs_S$, so $W$ witnesses the desired equality.
\end{proof}

In the light of the above additivity property, in order to establish the triviality of the bordism invariants of trivial colored links, it is enough to compute the invariants of simple loops.

\begin{lem}\label{lem:trivcobinv}
If $\Lc$ consists of a single unknotted loop, then 
\begin{equation}
\Bfr_{G,S}(\Lc, \psi) = 0
\end{equation}
in $\Omega^\SO_3(G, \Fs_S)$.
\end{lem}
\begin{proof}
By hypothesis, $\pi_1(S^3 \backslash \Lc) \cong \Zb$. Let $g = \psi(1) \in G$, and let $n$ be the order of $g$ in $G$. The $n$-fold cyclic cover $X$ of $S^3$, branched over $\Lc$, is diffeomorphic to $S^3$. Moreover, as the $C_n$-action on $X$ extends to a $C_n$-action on $\Db^4$, $X$ bords as a $C_n$-manifold. Moreover, $M(\Lc, \psi) \cong G \times_{C_n} X$ bords $G \times_{C_n} \Db^4$, and the non-trivial stabilizers of $\Db^4 \times_{C_n} G$ belong to $\Fs_S$, proving the desired equality\footnote{Let $H$ be a subgroup of $G$, and let $X$ be a $H$-space. Then $G \times_H X$ denotes the $G$ space obtained from the space of pairs $(g, x) \in G \times X$ with identification $(gh, x) \sim (g, h.x)$.}. 
\end{proof}

Next, we establish the conservation of the bordism invariant $\Bfr_{G,S}(\Lc, \psi)$ in topologically allowed local reconnections.

\begin{lem}\label{lem:reconncobinv}
Let $(\Lc, \psi)$ and $(\Lc', \psi')$ be $(G,S)$-colored links which can be transformed into each other by a single topologically allowed local reconnection (Fig.~\ref{fig:linkbasics}(c)). Then
\begin{equation}
\Bfr_{G,S}(\Lc, \psi) = \Bfr_{G,S}(\Lc', \psi')
\end{equation}
in $\Omega^\SO_3(G, \Fs_S)$.
\end{lem}
\begin{proof}
We will construct an equivariant oriented bordism between $M(\Lc, \psi)$ and $M(\Lc', \psi')$.

Let $D := \Db^3 \subset S^3$ be a small three-disk inside which the reconnection takes place, and let $C \subset D \times I$ be the saddle bordism depicted in Fig.~\ref{fig:saddlecob}. Then 
\begin{equation}
\Uc := \big(S^3 \backslash (\mathrm{Int}(D) \cup \Lc)\big) \times I \cup (D \times I \backslash C)
\end{equation}
is an oriented bordism between the complements of $\Lc$ and $\Lc'$ in $S^3$. We claim that the group homomorphism 
\begin{equation}
\psi^\circ : \pi_1\big(\big(S^3 \backslash (\mathrm{Int}(D) \cup \Lc)\big) \times I\big) = \pi_1\big(S^3 \backslash (\mathrm{Int}(D) \cup \Lc)\big) \to G
\end{equation}
extends to a group homomorphism
\begin{equation}
\tilde \psi: \pi_1(\Uc) \to G 
\end{equation}
satisfying
\begin{align}
\tilde \psi\vert_{(S^3 \backslash \Lc) \times \{0\}} &= \psi \\
\tilde \psi\vert_{(S^3 \backslash \Lc') \times \{1\}} &= \psi'.
\end{align}
Indeed, the restrictions of $\psi$ and $\psi'$ to $D \backslash \Lc$ and $D \backslash \Lc'$, respectively, allow us, by van Kampen's theorem, to construct a homomorphism
\begin{equation}
\bar \psi: \pi_1\Big(\big(S^3 \backslash (\mathrm{Int}(D) \cup \Lc)\big) \times I \cup (D \backslash \Lc) \times \{0\} \cup (D \backslash \Lc') \times \{1\}\Big)  \to G 
\end{equation}
restricting to the desired homomorphisms at $0$ and $1$. The complement of the saddle $C$ in $\partial D \times I \cup D \times \{0\} \cup D \times \{1\} \simeq S^3$ is a deformation retract\footnote{A \emph{deformation retract} is a (necessarily injective) continuous map $i: X \to Y$ such that there exists a continuous map $r: Y \to X$ satisfying $r \circ i = \Id_X$ (such an $r$ is called a \emph{retraction}), such that $i \circ r$ is homotopic to $\Id_Y$ relative to $X$.} of the complement of $S$ in $D \times I \cong \Db^4$, because the inclusion is homeomorphic to the standard inclusion of $\partial \Db^4 \backslash \partial \Db^2 \hook  \Db^4 \backslash \Db^2$. Hence,  we get a homotopical  identification
\begin{equation}
\big(S^3 \backslash (\mathrm{Int}(D) \cup \Lc)\big) \times I \cup (D \backslash \Lc) \times \{0\} \cup (D \backslash \Lc') \times \{1\} \simeq \Uc.
\end{equation}
As a homotopy equivalence induces an isomorphism on fundamental groups, we have found the desired group homomorphism $\tilde \psi$.

Denote by $\tilde \Uc$ the $G$-principal bundle over $\Uc$ associated to the group homomorphism $\tilde \psi$. It completes to a unique branched covering $W$ of $S^3 \times I$ \cite{fox:1957, fox:1970}. The space $W$ has a canonical structure of an oriented smooth manifold with a boundary: it is an oriented bordism between $M(\Lc, \psi)$ and $M(\Lc', \psi')$. Moreover, the  group $G$ acts on $W$ it by orientation-preserving diffeomorphisms, and this action restricts to the usual actions on $M(\Lc, \psi)$ and $M(\Lc', \psi')$. As the stabilizers of $W$ belong to $\Fs_S$, $W$ realizes the desired equality in the equivariant bordism group $\Omega^\SO_3(G, \Fs_S)$.
\end{proof}

The results of this section may be concisely summarized as the following theorem.

\begin{thm}\label{thm:consofcobinv}
Let $G$ be a finite group and $S \subset G$ be a subset that is closed under inverses and conjugation by elements of $G$. Then the bordism invariant $\Bfr_{G,S}$ satisfies the following properties:
\begin{enumerate}
\item \emph{conservation in reconnections:} if $(\Lc, \psi)$ and $(\Lc', \psi')$ are $(G,S)$-colored links that can be transformed into each other by a sequence of color respecting smooth isotopies\footnote{Diagrammatically, these can be expressed as sequences of colored Reidemeister moves \cite{rolfsen:2003, annala:2022}.} and topologically allowed local reconnections, then $\Bfr_{G,S}(\Lc, \psi) = \Bfr_{G,S}(\Lc', \psi')$;

\item \emph{annihilation of trivial links:} if $(\Lc, \psi)$ is a $(G,S)$-colored link such the link structure of $\Lc$ is trivial, then $\Bfr_{G,S}(\Lc, \psi) = 0$;

\item \emph{additivity:} $\Bfr_{G,S}\big( (\Lc_1, \psi_1) \uplus (\Lc_2, \psi_2) \big) = \Bfr_{G,S}( \Lc_1, \psi_1) +  \Bfr_{G,S}(\Lc_2, \psi_2)$. 
\end{enumerate}
\end{thm}

\section{Classification of topologically protected tricolorings}\label{sect:tricolor}

Here, we investigate the topological stability of tricolored links. Let $\Sigma_3$ denote the group of permutations of three elements, and let $I := \{ \sigma_{12}, \sigma_{13}, \sigma_{23} \} \subset \Sigma_3$ be the set of inversions, where $\sigma_{ij}$ denotes the permutation that swaps the $i^{th}$ and the $j^{th}$ element. 

\begin{defn}[cf. \cite{fox:1970}]\label{def:tricolo}
A \emph{tricolored link} is a $(\Sigma_3, I)$-colored link in the sense of Definition \ref{def:cobinv}. The data of a $(\Sigma_3, I)$-colored link diagram admits a graphical presentation, as a link diagram in which each arc is colored with either red, gray, or blue color, as explained in Fig.~\ref{fig:tricolorules}(a) and (b).
\end{defn}

\begin{defn}[cf. \cite{annala:2022}]
Two tricolored links are \emph{topologically equivalent}, if one of them can be transformed into the other using smooth isotopies and \emph{topologically allowed local surgeries} (Fig.~\ref{fig:tricolorules}(c) and (d)). A tricolored link is \emph{topologically trivial} if it is topologically equivalent to an untangled disjoint union of simple loops. A tricolored link that is not topologically trivial is \emph{topologically protected}. An example of a topologically trivial knot is  depicted in Fig.~\ref{fig:tricoloex}(a).
\end{defn}

The existence of topologically protected tricolored links is a non-trivial question: indeed, in order to establish the topological protection even of a single example, one needs to be able to detect a topologically protected knot. For this purpose we employ the bordism invariant developed in Section \ref{sect:cobinv}.

\begin{defn}
Let $(\Lc, \psi)$ be a tricolored knot. We define its \emph{$i$-invariant} as 
\begin{equation}
i(\Lc, \psi) := [M(\Lc, \psi)] \in \Omega^\SO_3(C_3, *),
\end{equation}
where $*$ is the family of subgroups of $C_3$ consisting of the trivial subgroup $\{e\}$. In other words, we restrict the $\Sigma_3$-action on $M(\Lc, \psi)$ into a $C_3$-action, and the resulting action is free because the only non-trivial stabilizers of the original $\Sigma_3$-action are the subgroups which are generated by a single inversion $\sigma_{ij}$.
\end{defn}

\begin{lem}\label{lem:isurgery}
The $i$-invariant is conserved in topologically allowed local surgeries.
\end{lem}
\begin{proof}
In fact, all topologically allowed local surgeries of tricolored links can be expressed as a sequence of smooth isotopies and topologically allowed local reconnections (Fig.~\ref{fig:tricolorules}). As $i(\Lc, \psi)$ is the image of the bordism invariant $\Bfr_{\Sigma_3,I}(\Lc, \psi)$ under the restriction-of-action map $\Omega^\SO_3(\Sigma_3, I) \to \Omega^\SO_3(C_3, *)$, the claim follows from the conservation of the latter in topologically allowed local reconnections (Theorem~\ref{thm:consofcobinv}).
\end{proof}

Let us compute the target group of the $i$-invariant. We start by recording the following observation.

\begin{lem}\label{lem:orcobtohom}
Let $G$ be a finite group and $X$ a space that has the homotopy type of a CW complex. Then the map $\Omega^\SO_i(X) \to H_i(X; \Zb)$ defined by
\begin{equation}
[M \xto{f} X] \mapsto f_*([M]), 
\end{equation}
where $[M] \in H_i(M; \Zb)$ is the fundamental class of the oriented manifold $M$, is an isomorphism in degrees $i \leq 3$.
\end{lem}
\begin{proof}
The claim follows immediately from the Atiyah--Hirzebruch spectral sequence and from the computation of the oriented bordism group of a point.
\end{proof}

\begin{lem}\label{lem:invinZ3}
We have
\begin{equation}
\Omega^\SO_3(C_3, *) \cong H_3(C_3;\Zb) \cong \Zb_3
\end{equation}
where $H_i(C_3, \Zb)$ denotes the \emph{group homology} of $C_3$ with integer coefficients.
\end{lem}
\begin{proof}
Since the $C_3$-action on $M(\Lc, \psi)$ is free, the map $M(\Lc, \psi) \to M(\Lc, \psi)/C_3$ is a principal $C_3$-bundle, and therefore it defines a continuous map $M(\Lc, \psi)/C_3 \to \mathrm{B}C_3$, unique up to homotopy, where $\mathrm{B}C_3$ is the classifying space of the group $C_3$. In this fashion, one gets an isomorphism $\Omega^\SO_3(C_3, *) \cong \Omega^\SO_3(\mathrm{B}C_3)$ \cite{conner-floyd:1964}. By Lemma \ref{lem:orcobtohom}, $\Omega^\SO_3(\mathrm{B}C_3) \cong H_3(\mathrm{B}C_3; \Zb) = H_3(C_3; \Zb)$. The isomorphism $H_3(C_3;\Zb) \cong \Zb_3$ is well known. 
\end{proof}

By Theorem \ref{thm:consofcobinv} and Lemma \ref{lem:isurgery} the $i$-invariant is trivial for a tricolored link that is not topologically protected. Below, we give examples of topologically protected tricolored knots. 

\begin{thm}\label{thm:nontrivtrefoil}
Let $(\Lc, \psi)$ be a tricolored trefoil knot (Fig.~\ref{fig:tricoloex}(b)). Then $i(\Lc, \psi) \not = 0 \in \Zb_3$.
\end{thm}
\begin{proof}
We have to show that, for a tricolored trefoil $(\Lc, \psi)$, the induced map 
\begin{equation}
f: X := M(\Lc,\psi)/C_3 \to \mathrm{B}C_3
\end{equation}
satisfies $f_*[X] \not = 0 \in H_3(\mathrm{B}C_3; \Zb) \cong \Zb_3$. As $[X]$ is a generator of $H_3(X; \Zb) \cong \Zb$, this is equivalent to the surjecitivy of $f_*: H_3(X;\Zb) \to H_3(\mathrm{B}C_3; \Zb)$. 

By Ref.~\cite[Chapter 10D]{rolfsen:2003}, $X$ is diffeomorphic to the lens space $L(3,1)$, which in turn is diffeomorphic to $S^3$ by a free action on $C_3$. By construction, $\pi_1(f): \pi_1(X) \to \pi_1(\mathrm{B}C_3)$ is surjective; as both groups are of order 3, $\pi_1(f)$ is an isomorphism. In conclusion, we have a homotopy fiber sequence of spaces
\begin{equation}\label{eq:trefoilfibseq}
S^3 \to X \to \mathrm{B}C_3.
\end{equation}

The form of the Serre spectral sequence (see e.g. \cite{hatcher:ss, boardman:1999}) of the fiber sequence (\ref{eq:trefoilfibseq}) is
\begin{equation}
E^2_{p,1} = H_p(\mathrm{B}C_3;H_q(S^3;\Zb)) \Rightarrow H_{p + q}(X; \Zb),
\end{equation}
and the differentials are homomorphisms $d_r: E^r_{p,q} \to E^r_{p-r,q+r-1}$. %Diagrammatically, the $E^2$-page of the spectral sequence  is 
%\[
%\begin{tikzcd}
%H_0(\mathrm{B}C_3; \Zb) & H_1(\mathrm{B}C_3; \Zb) & H_2(\mathrm{B}C_3; \Zb) & H_3(\mathrm{B}C_3; \Zb) & \cdots \\
%0 & 0 & 0 & 0 & \cdots \\
%0 & 0 & 0 & 0 & \cdots \\
%H_0(\mathrm{B}C_3; \Zb) & H_1(\mathrm{B}C_3; \Zb) & H_2(\mathrm{B}C_3; \Zb) \arrow[llu]{}{d_2} & H_3(\mathrm{B}C_3; \Zb) \arrow[llu]{}{d_2} & \cdots 
%\end{tikzcd}
%\]
The desired surjectivity of the homomorphism $f_*$ is equivalent to the equality 
\begin{equation}
E^2_{3,0} = E^\infty_{3,0}.
\end{equation}
As the cells $E^i_{3,0}$ do not support any non-trivial differentials for dimensional reasons, the claim follows.
\end{proof}

\begin{rem}\label{rem:norminv}
As left-handed and right-handed trefoils are equivalent to each other under an orientation reversing diffeomorphism, the left-handed and the right-handed trefoils have opposite $i$-invariants. We normalize the $i$-invariant in such a way that the $i$-invariant of the right-handed trefoil is by definition $[1] \in \Zb_3$.
\end{rem}

We have shown the existence of two non-equivalent topologically protected tricolored links. Moreover, any tricolored link that is equivalent to an untangled disjoint union of tricolored trefoils is either topologically trivial, or equivalent to a single tricolored trefoil (Fig.~\ref{fig:tricolotable}). Next we prove that every topologically protected tricoloring is topologically equivalent to a single trefoil.

\begin{thm}\label{thm:tricoloclassification}
Let $(\Lc, \psi)$ be a tricolored link. Then $(\Lc, \psi)$ is either topologically trivial, or topologically equivalent to either a left-handed or a right-handed tricolored trefoil.
\end{thm}
\begin{proof}
By the trefoil addition rules (Fig.~\ref{fig:tricolotable}), it is enough to show that $(\Lc, \psi)$ is topologically equivalent to an untangled disjoint union of simple loops and tricolored trefoils. Using the local surgery of Fig.~\ref{fig:tricoloclass1}(a), we can transform the the link diagram of $(\Lc, \psi)$ into a form, in which each crossing is immediately followed by another crossing. Denoting each double crossing by an edge, the link diagram is transformed into a diagram consisting of a disjoint union of simple loops that connected by edges (Fig.~\ref{fig:tricoloclass1}(b) and (c)). We will argue by induction on the number of edges.

Suppose $L$ is an \emph{innermost loop} in the diagram, i.e., inside $L$ there are no other loops. Using topologically allowed reconnections, we may assume that $L$ is connected to at most $3$ edges (Fig.~\ref{fig:tricoloclass1}(d)). Fig.~\ref{fig:tricoloclass2} and Fig.~\ref{fig:tricoloclass3} illustrate general procedures that reduce the number edges by performing local surgeries around $L$ in the cases where $L$ is connected to two and three edges, respectively. Hence, using topologically allowed local surgeries, $(\Lc, \psi)$ can be transformed into an untangled disjoint union of simple loops and tricolored trefoils, proving the claim.
\end{proof}

\section{Discussion}\label{sect:conclusions}

In Section~\ref{sect:cobinv}, we constructed bordism invariant of colored links, which we employed in Section~\ref{sect:tricolor} to classify tricolored links up to topologically allowed local surgeries. In physical terms, tricolored links model those topological vortex configurations supported by a medium, the order parameter space $X$ of which satisfies $\pi_1(X) \cong \Sigma_3$ and $\pi_2(X) \cong 0$, that do not contain vortices corresponding to cyclic permutations of order 3. Of course, if a given $G$-colored link diagram takes values in a subgroup $H \subset G$, then the same is true for any link diagram obtained from it by $G$-colored Reidemeister moves and topologically allowed local reconnections. Therefore, for physical realization, it would be enough to find a system for which $\Sigma_3 \subset \pi_1(X)$ and $\pi_2(X) \cong 0$. 

Are there any concrete physical systems for which $\pi_1(X)$ contains $\Sigma_3$ as a subgroup? Unfortunately, this cannot be the case for spinor Bose--Einstein condensates or liquid crystals, because $\Sigma_3$ does not appear as a discrete subgroup of $\SU(2)$ \cite{mckay:1981}. Rather, it is the corresponding double cover $\Dic_3$, \emph{the dicyclic group} of order 12, which appears as a subgroup of $\SU(2)$.  However, one can still leverage tricolored links to study the stability of vortex configurations: if there exists a (surjective) group homomorphism $\psi: G \to \Sigma_3$, and a $G$-colored link that maps under $\psi$ to a topologically protected tricolored link, then the original link is topologically protected as well. However, it is not clear how successful this strategy would be for $\Dic_3$: the nontrivial tricoloring of a trefoil does not lift to a $\Dic_3$-coloring because the group homomorphism $B_3 \to \Sigma_3$ that sends a braid on three strands to the induced permutation on three letters does not factor through the degree-two homomorphism $\Dic_3 \to \Sigma_3$. 

On a purely theoretical level, it is not difficult to come up with reasonable models, the order parameter space $X$ of which satisfies $\pi_1(X) \cong \Sigma_3$ and $\pi_2(X) \cong 0$. For example, as $\Sigma_3$ is a discrete subgroup of the simply connected Lie group $\SU(6)$, and as, by Mostow--Palais theorem \cite{mostow:1957, palais:1957}, there exists a spontaneous-symmetry-breaking process that breaks $\SU(6)$ down to $\Sigma_3$, there exists a Yang--Mills theory such that the order parameter space of the vacuum structures is homeomorphic to $\SU(6) / \Sigma_3$. It might be possible to realize such systems in laboratory setting as artificial gauge field theories \cite{dalibard:2011}.

The results of this article naturally lead to further questions. Consider for example the $\Zb_n$-valued invariants for Fox $n$-colorings \cite{fox:1970}, obtained from the bordism invariant of Section~\ref{sect:cobinv} similarly to how the $i$-invariant was obtained in Section~\ref{sect:tricolor}. Does this invariant provide a complete classification of $n$-colorings in the sense that, for each non-trivial value $[m] \in \Zb_m$ of the invariant, there exists only one Fox $n$-coloring up to topological equivalence? For an odd prime $p > 2$, the torus knot $T_{2,p}$ admits a Fox $p$-coloring \cite{breiland:2009}. Moreover, since the branched double cover of $T_{2,p}$ is equivalent to the lens space $L(p,1)$ \cite{rolfsen:2003}, one shows as in the proof of Theorem \ref{thm:nontrivtrefoil} that the invariant of the $p$-coloring is a generator of $\Zb_p$. Hence, by additivity, there exist a Fox $p$-colored links that have every possible value of the invariant. 

\clearpage

\begin{figure}[H]
\includegraphics[scale=1.4]{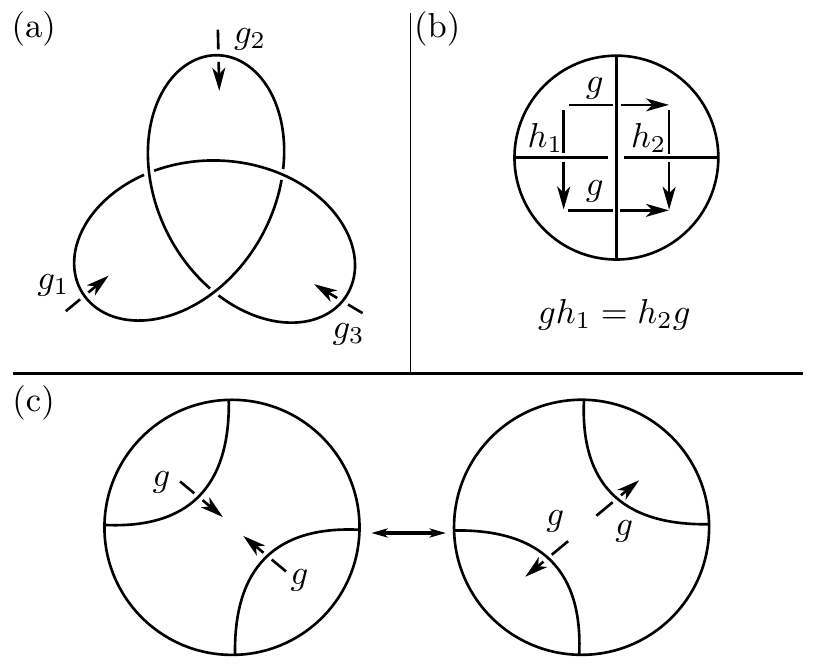}
\caption{ 
$G$-colored link diagrams and topologically allowed reconnections.
(a) In a $G$-colored link diagram, each arc is equipped with an arrow and an element $g \in G$. The content of this extra data is to present a group homomorphism $\psi: \pi_1(S^3 \backslash \Lc, s_0) \to G$ where $\Lc$ denotes the link and the basepoint $s_0$ is taken to lie above the plane of the paper. Namely, if $\gamma$ is the loop that starts at $s_0$, travels down to the back of the arrow, follows the arrow under the arc, and returns back to $s_0$, then $\psi([\gamma]) = g$. 
This rule defines a well-defined group homomorphism $\psi: \pi_1(S^3 \backslash \Lc, s_0) \to G$ if and only if the \emph{Wirtinger relation} (b) is satisfied at each crossing of the diagram. The direction of the arrows is completely arbitrary: the direction of any arrow may be reversed, as long as the corresponding element of $G$ is replaced by its inverse.
(c) A \emph{topologically allowed local reconnection} is a reconnection of the vortex cores that respects the $G$-coloring.
}\label{fig:linkbasics}
\end{figure}

\begin{figure}[H]
\includegraphics[scale=1]{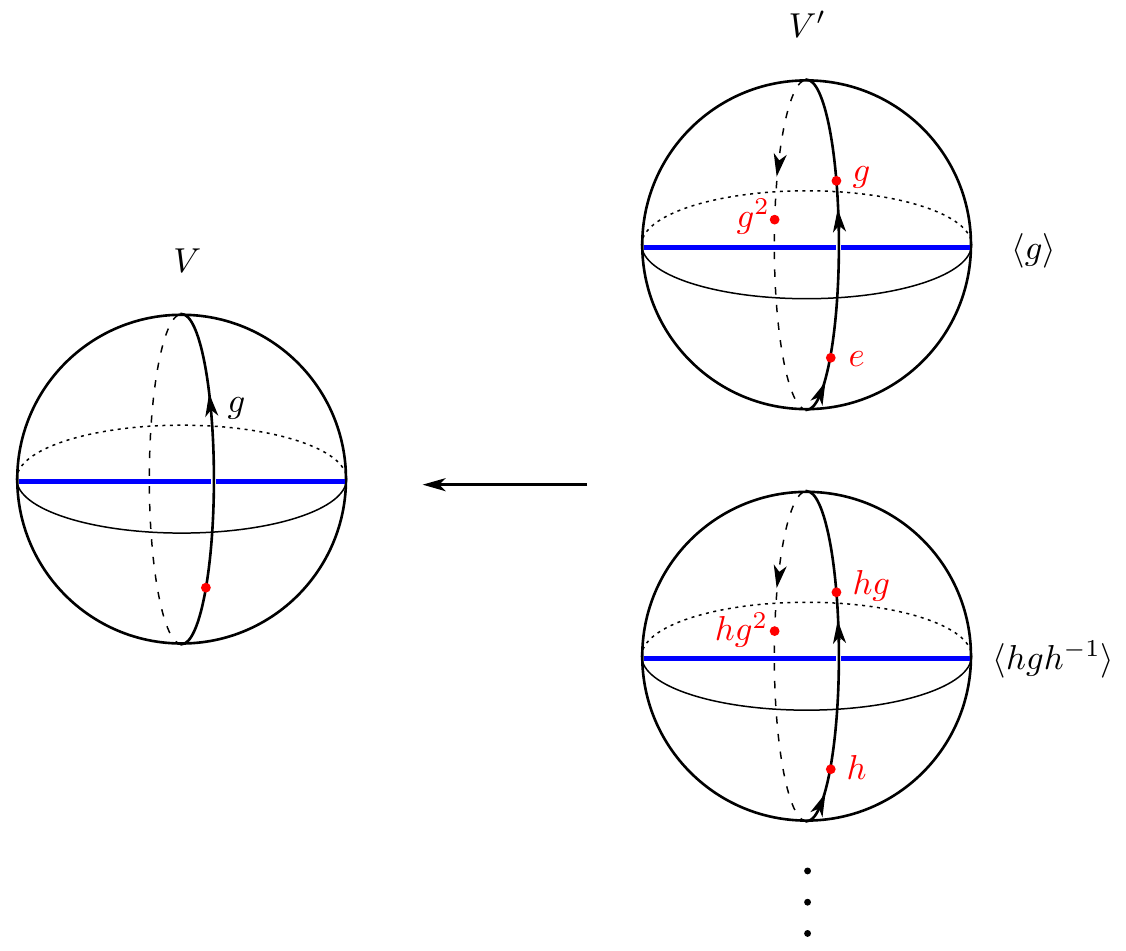}
\caption{ 
Local structure of the branched covering $M(\Lc, \psi) \to S^3$ (Construction~\ref{cons:M}). On the left, we have a small closed neighborhood $V$ of a point on the link $\Lc$ (blue). The homotopy class of the loop $\gamma$ winding about the core of $\Lc$ once (the basepoint of $\gamma$ is denoted by a red dot) is sent to $g \in G$ under the $G$-coloring $\psi$. On the right, we have the preimage $V'$ of $V$ under the branched covering $M(\Lc, \psi) \to S^3$. The preimages of the points in $\Lc$ are denoted by blue. On the complement of these points, the branched covering is a principal $G$-bundle. The preimages of the basepoint are denoted by red dots; there exist one such point for every element of $G$. Moreover the unique lift  of the loop $\gamma$, starting at the point corresponding to $h \in G$, ends at the point corresponding to $hg \in G$. The group $G$ acts on $M(\Lc, \psi) \vert_U$; on the red dots, this action is identified with the action of $G$ on itself by left multiplication. The stabilizer group of the blue points on each component is always conjugate to the cyclic group $\langle g \rangle$ generated by $g$. The stabilizer group is given for the two components depicted in the figure. The components of $M(\Lc, \psi) \vert_U$ are in bijective correspondence with the cosets $G/\langle g \rangle$.
}\label{fig:bcover}
\end{figure}

\begin{figure}[H]
\includegraphics[scale=1.1]{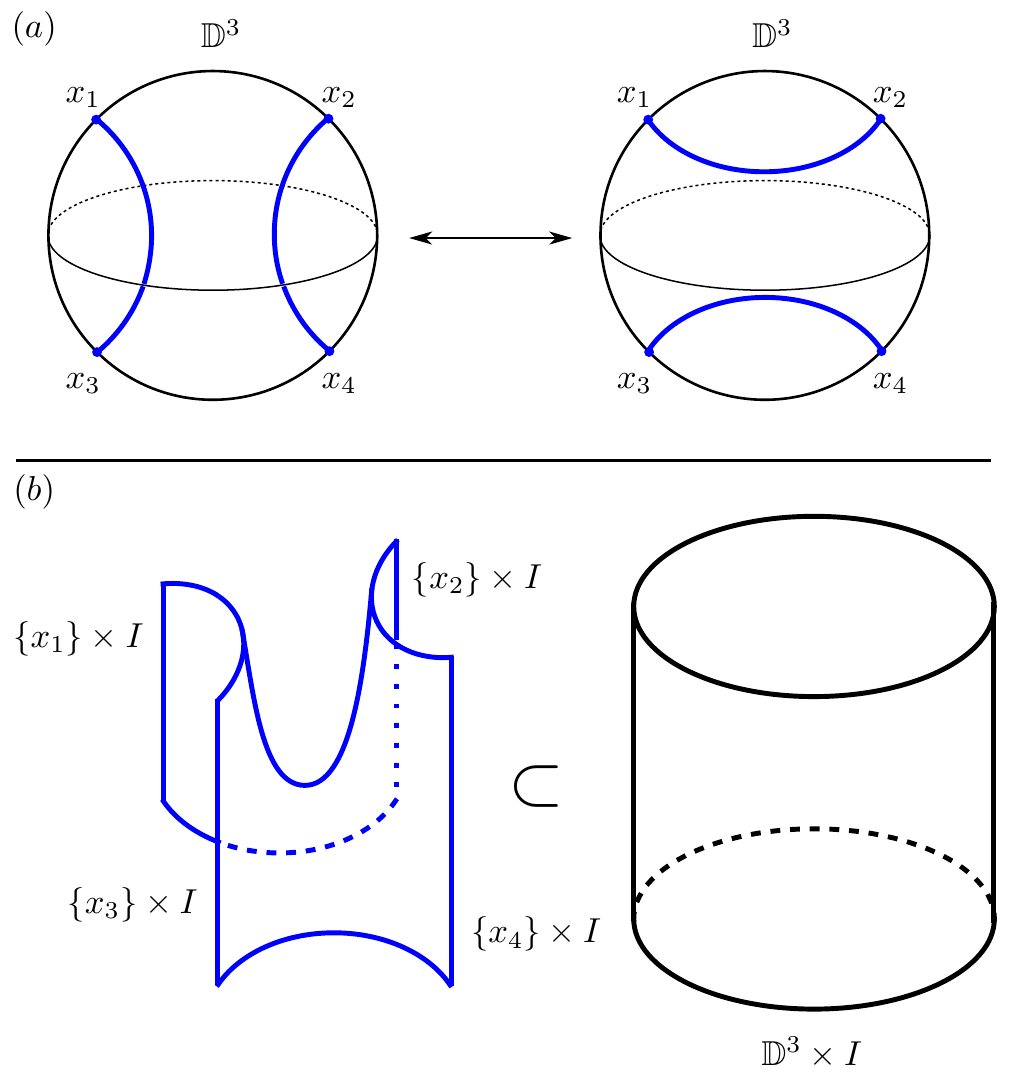}
\caption{
The saddle bordism (proof of Lemma \ref{lem:reconncobinv}).
(b) The complement of the saddle (blue) in $\Db^3 \times I$ provides a bordism between the complements of the two configurations of arcs (blue) inside $\Db^3$ (a). Both arc configurations have the same endpoints $x_i$ on the boundary sphere $\partial \Db^3$, and the saddle is constant on the boundary, i.e., $S \cap (\partial \Db^3 \times I) = \{x_1, x_2, x_3, x_4\} \times I$. 
}\label{fig:saddlecob}
\end{figure}

\begin{figure}[h!]
\includegraphics[scale=1.1]{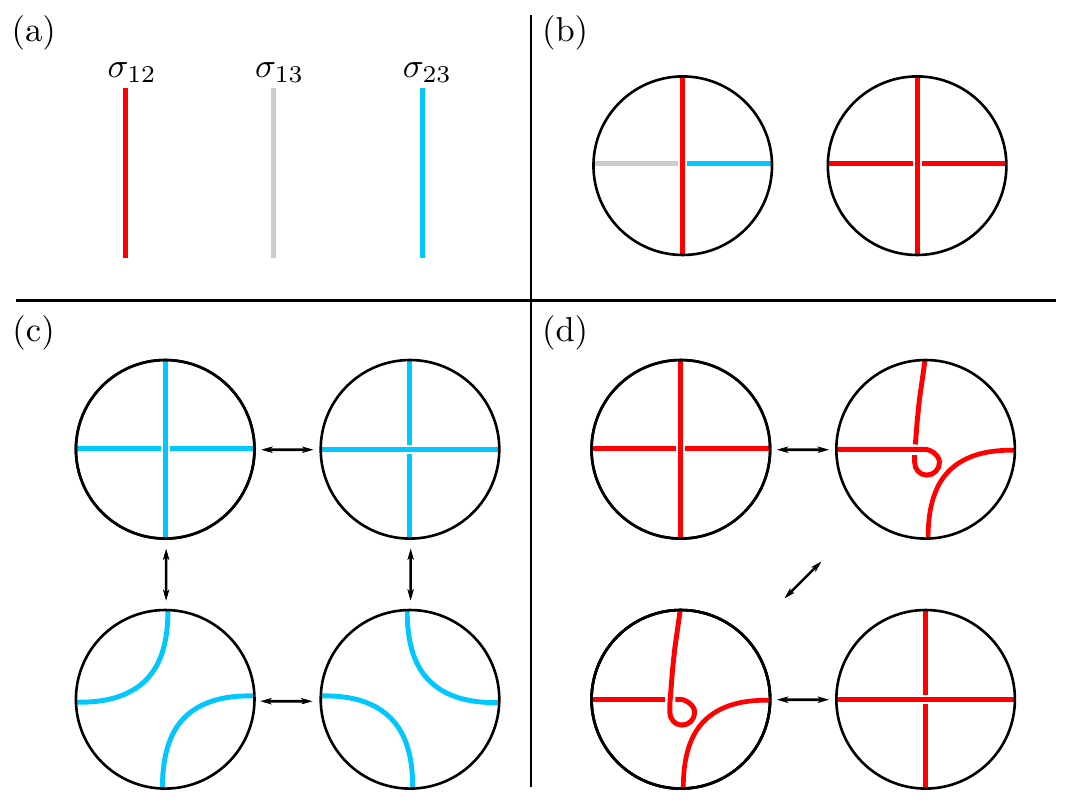}
\caption{
Graphical presentation of tricolored links (Definition~\ref{def:tricolo}). 
(a) A tricolored link diagram admits a more graphical presentation than that described in Fig.~\ref{fig:linkbasics}: instead of labeling each arc with an arrow and a group element, it suffices to color each arc with either red, green, or blue, which correspond to the inversions $\sigma_{12}, \sigma_{13}$, and $\sigma_{23}$, respectively.
(b) The Wirtinger relation admits a graphical presentation: at each crossing, either all arcs are of the same color, or all of them are of a different color.
(c) Topologically allowed local surgeries for tricolored links take place between strands of the same color. The surgery distinguished by a star is a \emph{saddle reconnection}.
(d) Up to link isotopy, all topologically allowed local surgeries can be reduced to a saddle reconnection. Above, this is demonstrated for the strand crossing.
}\label{fig:tricolorules}
\end{figure}

\begin{figure}[h!]
\includegraphics[scale=1.1]{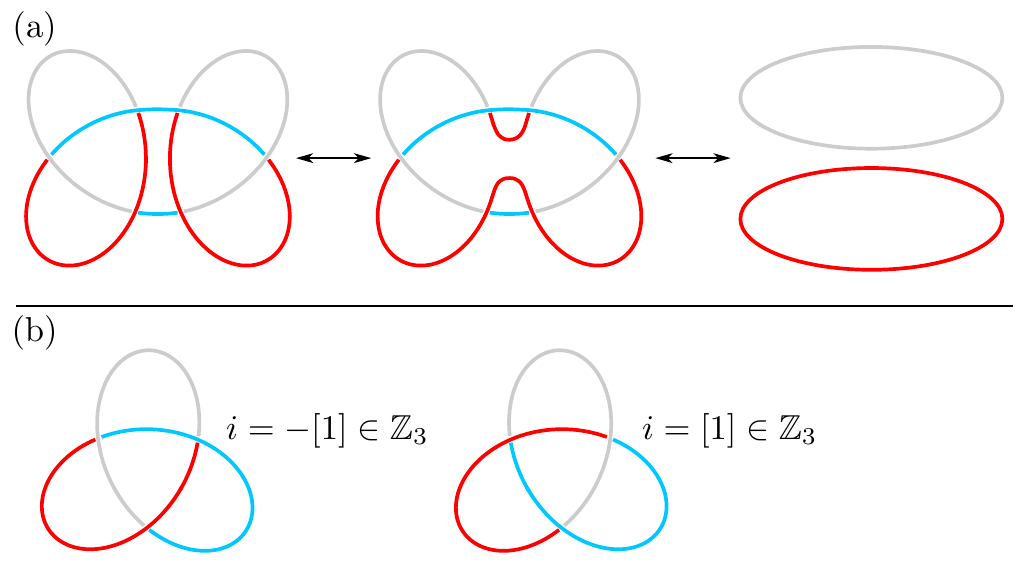}
\caption{
Examples of topologically trivial and topologically protected tricolored knots.
(a) An example of a topologically trivial tricolored knot, and an example of its decay.
(b) Two examples of topologically protected tricolored knots (Theorem \ref{thm:nontrivtrefoil}), and the values of their $i$-invariants (with the normalization convention of Remark~\ref{rem:norminv}). Every topologically protected link can be transformed into one of these structures using smooth isotopies and topologically allowed local surgeries (Theorem~\ref{thm:tricoloclassification}).
}\label{fig:tricoloex}
\end{figure}

\begin{figure}[h!]
\includegraphics[scale=1.1]{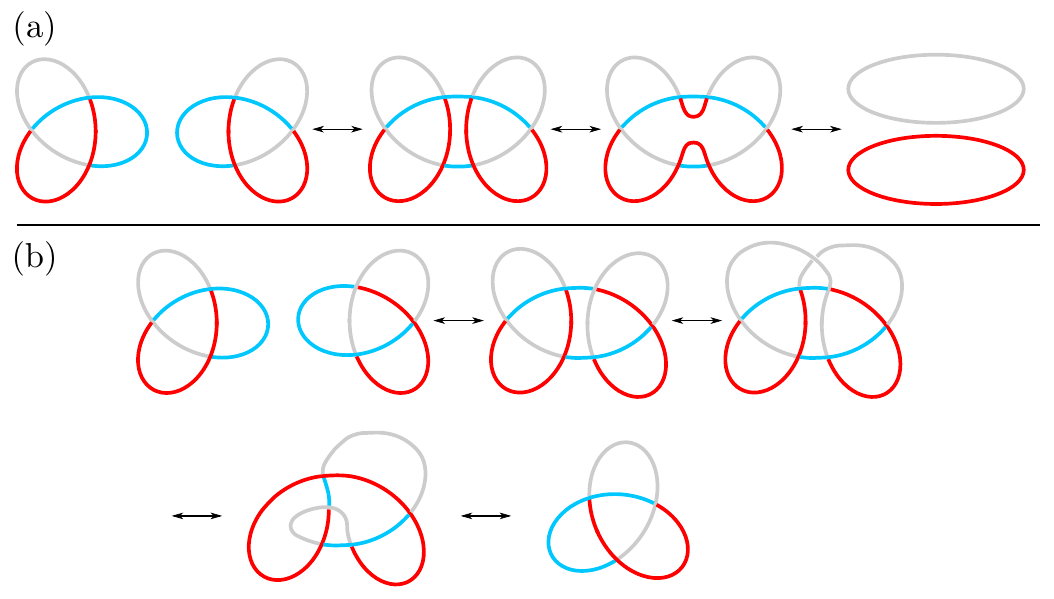}
\caption{
Addition rules for tricolored trefoils. 
(a) The combination of a left-handed and a right-handed tricolored trefoil can be transformed, using topologically allowed local surgeries, into two unlinked simple loops. 
(b) The combination of two left-handed tricolored trefoils can be transformed, using topologically allowed local surgeries, into a right-handed tricolored trefoil. Similarly, the combination of two right-handed tricolored trefoils can be transformed into a left-handed trefoil.
}\label{fig:tricolotable}
\end{figure}

\begin{figure}
\includegraphics[scale=1.1]{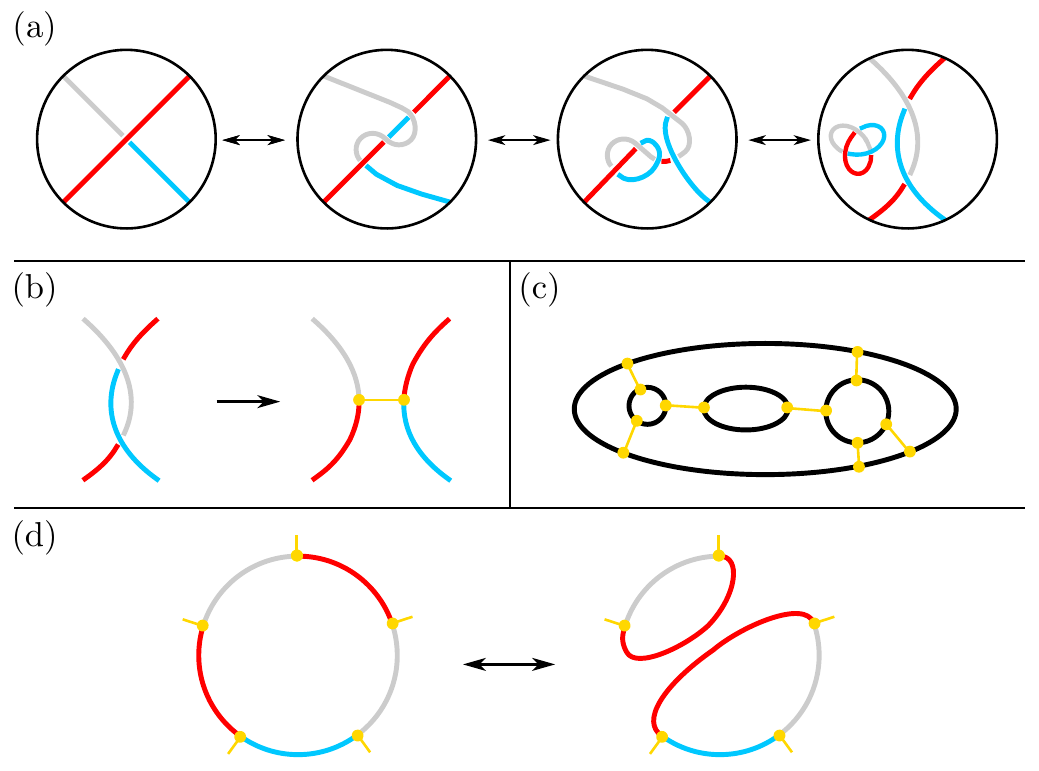}
\caption{
Transforming a tricolored link into a special form (proof of Theorem \ref{thm:tricoloclassification}). 
(a) By performing topologically allowed local surgeries, each crossing of the link diagram can be transformed into a form where an overcrossing is immediately followed by an undercrossing. A tricolored trefoil is split off from the link in the process.
(b) To simplify the graphical presentation, we depict each of the obtained double crossings by a yellow edge connecting the arcs.
(c) Using the notation described in (b), the surgeries depicted in (a) may be employed to transform a tricolored link diagram into a diagram consisting of non-overlapping simple loops, which are connected by edges. Here we have omitted the coloring of the arcs for simplicity.
(d) An innermost loop (a loop inside which there are no other loops) that is connected to more than three edges may be broken up to several loops, which are connected to at most three edges.
}\label{fig:tricoloclass1}
\end{figure}

\begin{figure}
\includegraphics[scale=0.9]{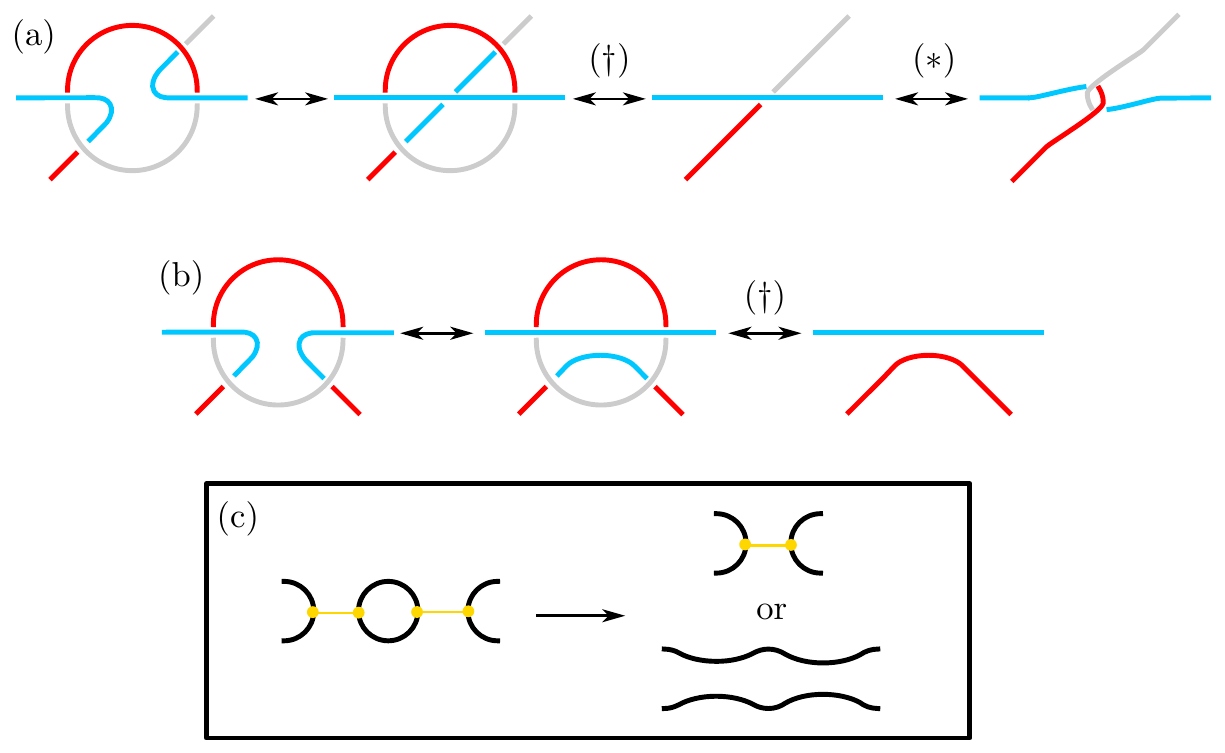}
\caption{
Eliminating an innermost loop connected to two edges (proof of Theorem \ref{thm:tricoloclassification}). 
(a),(b) Up to permutation of colors and isometries of the plane of the diagram, there are only two possible structures around an innermost loop that is connected to two edges. (c) In both cases, the number of edges may be reduced by at least one, with the cost of potentially splitting off a simple loop $(\dagger)$ or a tricolored trefoil knot $(*)$.
}\label{fig:tricoloclass2}
\end{figure}

\begin{figure}
\includegraphics[scale=1.05]{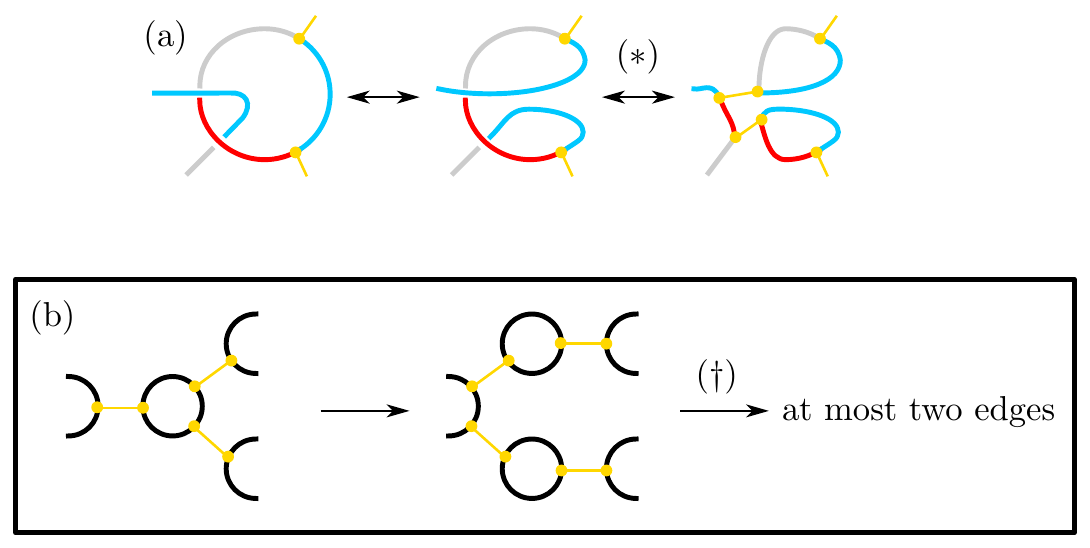}
\caption{
Eliminating an innermost loop connected to three edges (proof of Theorem \ref{thm:tricoloclassification}).
(a) Up to permutation of colors and isometries of the plane of the diagram, there is only one relevant structure to consider. The mark $(*)$ implies a process in which two trefoils are split off. (b) After performing the surgery that is depicted in (a), we may reduce to the two-edge case $(\dagger)$, which was considered in Fig.~\ref{fig:tricoloclass2}. Hence, the number of edges may be reduced by at least one, with the cost of potentially splitting off simple loops and tricolored trefoils.
}\label{fig:tricoloclass3}
\end{figure}

\clearpage

\bibliography{../../../References/references}


\end{document}